\documentclass{amsart}
\usepackage{amssymb}
\usepackage{centernot}

\newtheorem{thm}{Theorem}[section] 

\newtheorem{cor}[thm]{Corollary}

\newtheorem{lem}[thm]{Lemma}

\def\[{\left[}
\def\]{\right]}

\def\Ker{{\operatorname{Ker}}}
\def\gr{{\operatorname{gr}}}

\def\gr{{\operatorname{gr}}}

\begin{document}

\title{Affine Noetherian algebras, filtrations and presentations}

\begin{abstract} 
Resco and Small gave the first example of an affine Noetherian algebra which is not finitely presented. It is shown that their algebra has no finite-dimensional filtrations whose associated graded algebras are Noetherian, affirming their prediction. A modification of their example yields countable fields over which `almost all' (that is, a co-countable continuum of) affine Noetherian algebras lack such a filtration, and an answer to a question suggested by Irving and Small is derived.
\end{abstract}

\author{Be'eri Greenfeld}

\address{Department of Mathematics, University of California, San Diego, La Jolla, CA, 92093, USA}
\email{bgreenfeld@ucsd.edu}


\dedicatory{Dedicated to the memory of my erudite teacher A. S. Dahari}

\maketitle

\section{Introduction}

Let $F$ be a field and let $R$ be an $F$-algebra with a finite-dimensional filtration $F\subseteq \mathcal{F}_0\subseteq \mathcal{F}_1\subseteq \cdots$, that is, for all $i,j$ we have $\dim_F \mathcal{F}_i < \infty$  and $\mathcal{F}_i\mathcal{F}_j\subseteq \mathcal{F}_{i+j}$. We always assume that our filtrations are exhaustive, namely, $R=\bigcup_{i=0}^{\infty} \mathcal{F}_i$. It is well known that if the associated graded algebra:
$$ \gr_\mathcal{F} (R) = \mathcal{F}_0 + \frac{\mathcal{F}_1}{\mathcal{F}_0} + \frac{\mathcal{F}_2}{\mathcal{F}_1} + \cdots $$ is Noetherian then $R$ is Noetherian as well, but the converse is false. It was asked by McConnell and Robson in 1987 whether every affine Noetherian algebra admits \textit{any} finite-dimensional filtration with respect to which the associated graded algebra is Noetherian (aka `Noetherian filtration'), see \cite[8.3.10]{MR0}. The first counterexample to this question was given by Stephenson and Zhang \cite{SteZh}; they proved that graded Noetherian algebras have subexponential growth, and hence any affine Noetherian algebra of exponential growth forms a counterexample to the aforementioned question. Later, counterexamples were found among PI-algebras \cite{StaZh}, which are thus of polynomially bounded growth.

In 1992, Resco and Small \cite{RescoSmall} gave the first example of an affine Noetherian algebra which does not remain Noetherian under a suitable scalar extension; interestingly, over an uncountable algebraically closed field such examples do not exist \cite{Bell} (for more on the importance of Noetherian stability under scalar extensions, see \cite{ASZ}; there are plenty of non-affine examples, for instance, if $K/F$ is a non-finitely generated field extension then $K\otimes_F K$ is not Noetherian \cite{Vamos}).  Resco and Small also proved that their example is non-finitely presented. To the best of our knowledge, the only known examples of affine Noetherian algebras having these two unexpected properties are of the kind of the Resco-Small algebras.
Motivated by the McConnell-Robson problem, Resco and Small computed the associated graded algebra of their ring for various filtrations and asked if it admits a Noetherian filtration. We remark that their algebra has subexponential (in fact, intermediate) growth, so the Stephenson-Zhang argument does not apply.

\begin{thm} \label{noeth_fp}
Let $R$ be an algebra and let $\mathcal{F}$ be a finite-dimensional filtration. If $\gr_\mathcal{F}(R)$ is Noetherian then $R$ is finitely presented.
\end{thm}

As a consequence, we affirm the suspicion of Resco and Small and answer their question:

\begin{cor} \label{rescosmall}
The Resco-Small algebra has no finite-dimensional filtration whose associated graded algebra is Noetherian.
\end{cor}

In 1986, Irving and Small \cite{IrvingSmall} investigated necessary conditions for representability of PI-algebras, and as a consequence of their work they proved that over a countable field there are only countably many isomorphism classes of affine Noetherian PI-algebras (in retrospect, this can be deduced also from Anan'in's representability result \cite{Ananin} or from the fact that affine Noetherian PI-algebras are finitely presented \cite[Theorem~10.1.2]{BKR}, both proven later). They suggested that the following question is reasonable: over a countable field, are there only countably many isomorphism classes of affine Noetherian (not necessarily PI) algebras?
Using a further modification of the Resco-Small argument, we resolve this question in the negative:

\begin{thm} \label{main_1}
There exists a countable field $F$ and a continuum of pairwise non-isomorphic affine Noetherian $F$-algebras.
\end{thm}

Combined with Theorem \ref{noeth_fp}, we get that over certain fields, `almost all' affine Noetherian algebras do not afford a finite-dimensional Noetherian filtration:

\begin{cor} \label{almost}
There exists a countable field over which there is a continuum of non-isomorphic affine Noetherian algebras without Noetherian filtrations, and only countably many non-isomorphic affine algebras which admit a Noetherian filtration.
\end{cor}

This gives (another) resounding negative answer to McConnell-Robson's question on the existence of Noetherian filtrations.

\section{Graded algebras}

We start with a slight generalization of a theorem of Lewin \cite[Theorem~17]{Lewin}. Lewin proved that if $U\triangleleft F\left<x_1,\dots,x_r\right>$ is an ideal of a free algebra which is homogeneous with respect to the standard grading and $F\left<x_1,\dots,x_r\right>/U$ is Noetherian then $U$ is finitely generated as a two-sided ideal. Lewin's proof works mutatis-mutandis for any grading of the free algebra, as long as it remains connected (that is, $S_0=F$):

\begin{thm}[{Lewin's Theorem for non-standard connected gradings}] \label{Lewin's} 
Let $S=F+S_1+\cdots$ be a finitely generated connected graded $F$-algebra and $\dim_F S_i<\infty$ for all $i\geq 1$. If $S$ is Noetherian then it is finitely presented.
\end{thm}

In fact, we will see that one can omit the connectedness assumption:

\begin{thm}[{An extended version of Lewin's theorem}] \label{extended}
Let $S=S_0+S_1+\cdots$ be a finitely generated graded $F$-algebra and $\dim_F S_i<\infty$ for all $i\geq 0$. If $S$ is Noetherian then it is finitely presented.
\end{thm}

The following seems to be well-known and definitely proven in various sources for commutative rings, but we give here a proof of this folklore for completeness.

\begin{lem} \label{gr_aff}
A non-negatively graded Noetherian $F$-algebra with finite-dimensional homogeneous components is affine over $F$. 
\end{lem}
\begin{proof}
Let $R=R_0+R_1+\cdots$ be a graded Noetherian algebra with $\dim_F R_i<\infty$. Suppose that $I=R_1+R_2+\cdots$ is generated as a left ideal by $R_1+\cdots+R_k$ (this is true for some $k$ by Noetherianity). We claim that each $R_m \subseteq F\left<R_0\cup\cdots\cup R_k\right>$. Work by induction. Pick $m>k$ and let $\alpha\in R_m$. Then $\alpha\in R\cdot \left(R_1+\cdots+R_k\right)$ so we can write: $$\alpha=\sum_{i_1} a_{1,i_1}\cdot r_{1,i_1}+\cdots+\sum_{i_k} a_{k,i_k}\cdot r_{k,i_k}$$ with each $a_{l,i_l}\in R$ homogeneous and $r_{l,i_l}\in R_l$. Since $\deg(\alpha)=m$, we may assume that all $a_{l,i_l}\in R_0\cup\cdots \cup R_{m-1}$, and the claim follows by the induction hypothesis.
\end{proof}


\begin{proof}[{Proof of Theorem \ref{extended}}]
Consider $T=F+S_1+S_2+\cdots\subseteq S$, a connected graded subalgebra. Obviously $S_0$ is a finite-dimensional subalgebra of $S$, so $T$ is a finite-codimensional subalgebra of $S$. By \cite[Lemma~1.4]{Stafford}, it follows that $T$ is Noetherian as well.
By Lemma \ref{gr_aff}, the algebra $T$ is also affine, say, generated by $t_1,\dots,t_n$. As a connected graded affine Noetherian algebra, by Theorem \ref{Lewin's} it follows that $T$ is finitely presented. Since $S$ is a finitely generated $T$-module, it is a homomorphic image of an algebra of the form $T'=T\left<e_1,\dots,e_p\right>/I$, where $I$ is the ideal generated by (finitely many) relations expressing each $e_ie_j$ and $e_it_j$ as a $T$-linear combination of $e_1,\dots,e_p$ (with coefficients given as polynomials in $t_1,\dots,t_n$). Then $T'$ is a finitely presented algebra (recalling that $T$ is) and a finitely generated $T$-module, hence Noetherian, so every homomorphic image of it -- in particular $S$ -- is finitely presented as well.
\end{proof}

\begin{lem} \label{fp}
Let $\mathcal{F}$ be a finite-dimensional filtration of an algebra $R$. If $\gr_\mathcal{F}(R)$ is finitely presented then so is $R$.
\end{lem}
\begin{proof} 
This is proven in \cite{Lorenz} for filtrations arising from powers of a generating subspace; let us generalize the argument. Fix a set of homogeneous generators of $\gr_\mathcal{F}(R)$, say, $a_1,\dots,a_n$ of degrees $d_1\leq\cdots\leq d_n$ respectively. Any system of lifts $\hat{a}_1,\dots,\hat{a}_n \in R$ of $a_1,\dots,a_n$, whose leading terms with respect to $\mathcal{F}$ agree with $a_1,\dots,a_n$, generates $R$. Let $A=F\left<x_1,\dots,x_n\right>$ be the $n$-generated free algebra. We have a presentation: $$ 0\rightarrow I \rightarrow F\left<x_1,\dots,x_n\right> \xrightarrow{\pi} R \rightarrow 0 $$ 
given by $\pi(x_i)=\hat{a}_i$. Now grade $A=\bigoplus_{d=0}^{\infty} A_d$ by $\deg(x_i)=d_i$ and let $\mathcal{G}$ be the corresponding filtration (namely, $\mathcal{G}_{d} = A_0 + \cdots + A_d$).
Notice that $\pi(\mathcal{G}_d)=\mathcal{F}_d$.
Consider the surjective morphism of graded algebras $\varphi \colon A\rightarrow \gr_{\mathcal{F}}(R)$ given by $\varphi(x_i)=a_i$. Now $\Ker(\varphi)$ is a homogeneous ideal of $A$. Let $f(x_1,\dots,x_n)\in \Ker(\varphi) \cap A_d$. Then $f(a_1,\dots,a_n)$ is zero in $\gr_\mathcal{F}(R)$, which means that in $R$ it holds that $f(\hat{a}_1,\dots,\hat{a}_n)\in \mathcal{F}_{d-1}$ (for $d=0$, formally set $\mathcal{F}_{-1}=\mathcal{G}_{-1}=0$); equivalently, $f(x_1,\dots,x_n)\in I+\mathcal{G}_{d-1}$.
It follows that:
$$ \Ker(\varphi) = I\cap \mathcal{G}_0 +
\frac{\left(I+\mathcal{G}_0\right)\cap \mathcal{G}_1}{\mathcal{G}_0}
+
\frac{\left(I+\mathcal{G}_1\right)\cap \mathcal{G}_2}{\mathcal{G}_1} + \cdots =: \gr_\mathcal{G}(I) $$
hence $\gr_\mathcal{F}(R)\cong F\left<x_1,\dots,x_n\right>/\gr_\mathcal{G}(I)$. 
The assignment $J\mapsto \gr_\mathcal{G}(J)$, carrying ideals of the free algebra to graded ideals is order-preserving and injective on chains (see \cite[Proposition~6.7]{MR}). Write $I=\bigcup_{i=1}^{\infty} I_i$ as an ascending union of finitely generated ideals. Then $\gr_\mathcal{G}(I)=\bigcup_{i=1}^{\infty} \gr_\mathcal{G}(I_i)$, but since $\gr_\mathcal{G}(I)$ 
is finitely generated (recall that finite presentation holds for any finite set of generators), both chains must stabilize. Hence $I$ is finitely generated as well, and $R$ is finitely presented.
\end{proof}

\begin{proof}[{Proof of Theorem \ref{noeth_fp}}]
Let $\mathcal{F}$ be a finite-dimensional filtration of $R$. Write $S:=\gr_\mathcal{F}(R)=S_0+S_1+\cdots$ and assume that it is Noetherian. By Theorem \ref{extended}, $S$ is finitely presented. By Lemma \ref{fp}, $R$ is finitely presented.
\end{proof}

Since the Resco-Small algebra is not finitely presented, it now evidently follows that it has no finite-dimensional filtrations whose associated graded algebras are Noetherian, thereby proving Corollary \ref{rescosmall}.


\section{A continuum of Noetherian algebras}

\begin{proof}[{Proof of Theorem \ref{main_1}}]
Fix a prime number $p$. Let $K=\mathbb{F}_p(t_0,t_1,\dots)$ be the field of rational functions over $\mathbb{F}_p$ in countably many algebraically independent variables. For each sequence $\lambda=(\lambda_0,\lambda_1,\dots)\in \mathbb{F}_p^{\infty}$, set a derivation $\delta=\delta_\lambda\colon K\rightarrow K$ by:
$$ \delta_\lambda(t_k)=t_{k+1}+\lambda_k t_0\ \ \text{for}\ k=0,1,\dots $$
This definition induces a well-defined derivation (e.g.~see \cite[Page 287, Exercise 5]{Rowen}).
Consider the associated differential polynomial ring $R=R_\lambda=K[x;\delta]$. Observe that $F=K^p\subseteq \Ker(\delta)$, so $F=\mathbb{F}_p(t_0^p,t_1^p,\dots)$ is a central subfield of $R$. Obviously $F$ is countable. Notice that $R$ is $F$-affine, and in fact $R=F\left<t_0,x\right>$. Indeed, assume by induction that $t_0,\dots,t_k\in F\left<t_0,x\right>$ for some $k\geq 0$; then: $$t_{k+1}=xt_k-t_kx-\lambda_kt_0\in F\left<t_0,x\right>.$$ Moreover, observe that $K$ is spanned over $F$ by the monomials:
$$ \{ t_{i_1}^{d_1}\cdots t_{i_r}^{d_r}\ |\ r\geq 1,\ i_1<\cdots<i_r,\ 0\leq d_1,\dots,d_r\leq p-1 \} $$
as an $F$-vector space. Hence $R=F\left<t_0,x\right>$. Since $K$ is a field, it immediately follows that $R$ is a Noetherian ring by \cite[Theorem~2.6]{GoodearlWarfield}. Hence for each sequence $\lambda\in \mathbb{F}_p^{\infty}$ we obtain an affine Noetherian $F$-algebra $R_\lambda$ as described above.

Fix $\lambda \in \mathbb{F}_p^{\infty}$. For any $f\in R_{\lambda}$ such that $f^p=t_k^p$, write $f=\sum_{i=0}^{m} c_ix^i$ with all $c_i\in K$. If $c_m\neq 0$ and $m>0$, it is easy to see that the highest non-zero term in $f^p$ is $c_m^px^{mp}$, so $m=0$ and $f$ lies in $K$. It is clear that the only $p$-th root of $t_k^p$ in $K$ is $t_k$ itself.

Fix $\lambda,\lambda' \in \mathbb{F}_p^{\infty}$. Let $\phi\colon R_\lambda \xrightarrow{\sim} R_{\lambda'}$ be an $F$-algebra isomorphism; in particular, $\phi(t_k^p)=t_k^p$. By the above argument, $\phi$ fixes $K$. Write $\phi(x)=\sum_{i=0}^{m} c_ix^i$ with all $c_i\in K$ and $c_m\neq 0$. Then: $$ K[x;\delta_{\lambda'}]=K+K\phi(x)+K\phi(x)^2+\cdots $$ so we can write $x=\sum_{j=0}^{r} d_j \phi(x)^j$ with all $d_j\in K$ and $d_r\neq 0$. The highest non-zero term of the right hand side of the last equation is then $d_rc_m^rx^{mr}$, so $mr=1$ and $m=1$. Now $\phi(x)=c_0+c_1x$ where $c_1\neq 0$. Then: $$ [\phi(x),t_k]=\phi(t_{k+1}+\lambda_k t_0)=t_{k+1}+\lambda_k t_0 $$
but: $$[\phi(x),t_k]=[c_1x,t_k]=c_1xt_k-t_kc_1x=c_1\delta_{\lambda'}(t_k)=c_1t_{k+1}+c_1\lambda'_kt_0$$
Since $\lambda'_k\in \mathbb{F}_p$, we always have that $t_{k+1}+\lambda'_kt_0\neq 0$, so:
$$ c_1 = \frac{t_{k+1}+\lambda_k t_0}{t_{k+1}+\lambda'_k t_0} $$
Take $k\gg 1$ such that $c_1\in \mathbb{F}_p(t_0,\dots,t_k)$, then: $$ 0 = \frac{\partial c_1}{\partial t_{k+1}} = \frac{\partial}{\partial t_{k+1}}\left(\frac{t_{k+1}+\lambda_k t_0}{t_{k+1}+\lambda'_k t_0}\right) = \frac{t_0(\lambda'_k-\lambda_k)}{(t_{k+1}+\lambda'_kt_0)^2} $$
hence $\lambda_k=\lambda'_k$.

Identify two sequences $\lambda \sim \lambda'$ if they agree for almost all slots. Since $|\mathbb{F}_p^{\infty}/\sim~|=2^{\aleph_0}$, the claim follows.
\end{proof}

\begin{proof}[{Proof of Corollary \ref{almost}}] 
This is straightforward by Theorem \ref{main_1} and Lemma \ref{fp}, as over a countable field there are only countably many isomorphism classes of finitely presented algebras.
\end{proof}

\section*{Acknowledgements} We thank Lance Small for inspiring related discussions.

\end{document}